\documentclass[3p,times]{elsarticle}

\usepackage{ecrc}

\volume{00}

\firstpage{1}

\journalname{}

\runauth{Chen Wang et al.}

\jid{procs}

\jnltitlelogo{}

\CopyrightLine{2024}{Published by Elsevier Ltd.}

\usepackage{amssymb}

\usepackage[figuresright]{rotating}

\usepackage{arydshln} 
\usepackage{amsmath}
\usepackage{amsthm}

\usepackage{graphics}
\usepackage{subfigure}
\usepackage{lineno,hyperref}
\usepackage{cases}
\usepackage{bm}
\usepackage[linesnumbered,ruled,vlined]{algorithm2e}
\usepackage{hyperref}
\usepackage[noend]{algpseudocode}
\modulolinenumbers[1]

\newtheorem{theorem}{Theorem}
\newtheorem{proposition}{Proposition}

\newtheorem{corollary}{Corollary}
\newtheorem{example}{Example}
\newtheorem{definition}{Definition}
\newtheoremstyle{prestyle}
{0} 
{\topsep} 
{\itshape} 
{} 
{\bfseries} 
{.} 
{.5em} 
{} 
\theoremstyle{prestyle}

\theoremstyle{definition}	

\begin{document}

\begin{frontmatter}



\dochead{}

\title{Determinants and Inverses of banded Toeplitz Matrices over $\mathbb{F}_p$ Are Periodic}


\author[]{Chen Wang}
\ead{2120220677@mail.nankai.edu.cn}
\author[]{Chao Wang\corref{cor1}}
\ead{wangchao@nankai.edu.cn}
\address[mymainaddress]{Address: College of Software, Nankai University, Tianjin 300350, China}

\cortext[cor1]{Corresponding author.}

\begin{abstract}
	
	Banded Toeplitz matrices over $\mathbb{F}_p$, as a well-known class of matrices, have been extensively studied in the fields of coding theory and automata theory. In this paper, we discover that both determinants and inverses of banded Toeplitz matrices over $\mathbb{F}_p$ exhibit periodicity. For a Toeplitz matrix with bandwidth $k$, The period $P(f)$ is related to the parameters on the band and is independent of the order, with an upper limit of $P(f) \le p^{k-1}-1$. We provide an algorithm which can compute the determinant of any order banded Toeplitz matrix within $O(k^4)$. And its inverse can be represented by three submatrices of size $P(f)*P(f)$ located respectively on the diagonal, above the diagonal, and below the diagonal. Thus, the computational cost for calculating the inverse is fixed, and our algorithm can solve it within $O(k^5)+3kP(f)^2$. This is the first time that the periodicity of determinants and inverses of general banded Toeplitz matrices over $\mathbb{F}_p$ has been computed and proven.
\end{abstract}


\end{frontmatter}

\section{Introduction and preliminaries \label{s1}}

Banded Toeplitz matrices, as shown in Eq. \ref{M}, are commonly used in linear systems obtained from the discretization of differential equations, and are widely applied in fields such as numerical analysis and signal processing, etc. \cite{ACETO20122960, ACETO20123857, poletti2021superfast}. Over finite fields, they are usually used as transition matrices for linear cellular automata, have been extensively studied in coding theory and automata theory etc. \cite{dumas2012computational, gray2006toeplitz, macwilliams1977theory, marti2011reversibility, silverman2008introduction}.
\begin{equation}
	\label{M}
	M_n = 
	\begin{bmatrix}
		
		c_{0} & \cdots & c_{R} & 0 & \cdots & 0\\
		\vdots & \ddots &  & \ddots & \ddots & \vdots\\
		c_{-L} &  & \ddots &  & \ddots & 0\\
		0 & \ddots &  & \ddots & & c_{R}\\
		\vdots & \ddots & \ddots &  & \ddots & \vdots\\
		0 & \cdots & 0 & c_{-L} & \cdots & c_{0}
		
	\end{bmatrix}_{n*n}.
\end{equation}

Fast computation methods for banded Toeplitz matrices have been extensively studied in the past \cite{ bini1999effective, heinig2011fast, meek1983inverses, trench1974inversion, trench1985explicit}. Currently, the complexity of computing the determinant has been optimized to a logarithmic scale of the order \cite{CINKIR20122298,cinkir2014fast,jia2016homogeneous,kilic2008computational,lv2008note}, while the complexity for computing the complete inverse has reached a quadratic scale \cite{hadj2008fast,LV20071189,lv2008note, ng2002inversion, wang2015explicit, zhao2008inverse}. However, these algorithms do not exhibit higher efficiency over finite fields. In fact, some algorithms that utilize fast Fourier transforms and iterators fail to work properly.

Papers \cite{del2013invertible,del2015note} provided the reversibility period for $k$-diagonal Toeplitz matrices whose elements all equal to 1 on $\mathbb{F}_2$, but the analysis of the properties of special matrices is difficult to generalize to general cases. Papers discovered the periodicity of banded Toeplitz matrices over $\mathbb{F}_p$ and provided a fast computation method for invertibility using DFA and polynomials \cite{DU2022163,YANG201523}. Additionally, the periodicity of banded circulant matrices  over $\mathbb{F}_p$ has been extensively studied \cite{Akin2021,cinkir2011reversibility,martin2015reversible}. However, banded Toeplitz matrices do not possess the same advantageous properties (operational closure) as banded circulant matrices which makes the computations more complex.

In this paper, we are the first to discover the periodicity of the determinant and inverse of general banded Toeplitz matrices over $\mathbb{F}_p$. For a Toeplitz matrix with bandwidth $k$, The period $P(f)$ is related to the parameters on the band and is independent of the order, with an upper limit of $P(f) \le p^{k-1}-1$. We provide an algorithm which can compute the determinant of any order banded Toeplitz matrix within $O(k^4)$. And the inverse can be represented by three submatrices of size $P(f)$ located respectively on the diagonal, above the diagonal, and below the diagonal. Thus, the computational cost for calculating the inverse is fixed, and our algorithm can solve it within $O(k^5)+3kP(f)^2$. This means that for Toeplitz matrices with limited bandwidths, we can efficiently compute their determinants and inverses over finite fields, regardless of their sizes.

This paper consists of four sections. In Section \ref{s2}, we perform formulas for determinants and inverses of banded Toeplitz matrix. Section \ref{s3} proves the periodicity of the determinant and inverse of banded Toeplitz matrices over finite fields and provides efficient computation methods. The final section summarizes the work of the entire paper.

\section{Fast computation for banded Toeplitz structure \label{s2}}

\subsection{Determinant of banded Toeplitz matrix}
Gaussian elimination is one of the most well-known algorithms for matrices, and it is particularly effective at handling band matrices: performing Gaussian elimination on an $n$-order matrix with a bandwidth of $k=L+1+R$ requires $O(k^2n)$ operations. However, the efficiency of Gaussian elimination does not improve significantly for banded Toeplitz matrices. Therefore, we introduce a variant of Gaussian elimination tailored for banded Toeplitz matrices, allowing for batch processing of Gaussian elimination on the same row or column. This can enhance the computational efficiency of matrices with a banded Toeplitz structure, reducing the complexity from linear to logarithmic with respect to the matrix order.

\begin{theorem}
	The determinant of an $n$-order banded Toeplitz matrix $M$ with parameters $c_{-L},\cdots,c_{R}$ satisfies the following formula:
	\begin{equation}
		\label{det}
		|M|=(-1)^{R(n-L-R)}c_{R}^{n-L-R}
		\det(T^{n-L-R}A | B )
	\end{equation}
	\begin{equation}
		\nonumber
		A = 
		\begin{bmatrix}
			c_{0} & \cdots & c_{R-1}\\
			\vdots & \ddots & \vdots\\
			c_{-L} & \cdots & c_{0}\\
			\vdots & \ddots & \vdots\\
			0 & \cdots & c_{-L}
		\end{bmatrix},
		B=
		\begin{bmatrix}
			c_{R} & \cdots & 0\\
			\vdots & \ddots & \vdots \\
			c_{0} & \cdots & c_{R}\\
			\vdots & \ddots & \vdots\\
			c_{-L+1} & \cdots & c_{0}
		\end{bmatrix},
		T=
		\begin{bmatrix}
			-\frac{c_{R-1}}{c_{R}} & 1 & 0  & \cdots & 0\\
			-\frac{c_{R-2}}{c_{R}} & 0 & 1  & \cdots & 0\\
			\vdots & \vdots & \ddots & \ddots & \vdots\\
			-\frac{c_{-L+1}}{c_{R}} & 0 & \cdots & 0 & 1\\
			-\frac{c_{-L}}{c_{R}}  & 0 & \cdots & 0 & 0
		\end{bmatrix}.
	\end{equation}
\end{theorem}

\begin{proof}
	In matrices with a Toeplitz structure, we can easily identify the following local construct: a column containing $c_{-L},\cdots,c_{R}$, with the remaining elements being zero, and the last row of the $(L+R+1)*d$ ($d \in \mathbb{Z}_+$) submatrix to its left being all zeros. Now, we use the column containing 
	$c_{-L},\cdots,c_{R}$ to perform Gaussian elimination on the left side, as shown in Eq. \eqref{gauss}.
	
	\begin{equation}
		\label{gauss}
		\begin{bmatrix}
			y_{1,1} & \cdots & y_{1,d} & c_{R}\\
			y_{2,1} & \cdots & y_{2,d} & c_{R-1}\\
			\vdots & \ddots\ & \vdots & \vdots\\
			y_{L+R,1} & \cdots & y_{L+R,d} & c_{-L+1}\\
			0 & \cdots & 0 & c_{-L}\\
		\end{bmatrix}
		\Longrightarrow 
		\begin{bmatrix}
			0 & \cdots & 0 & c_{R}\\
			y_{2,1}-\frac{y_{1,1}c_{R-1}}{c_{R}}  & \cdots & y_{2,d}-\frac{y_{1,d}c_{R-1}}{c_{R}} & c_{R-1}\\
			\vdots & \ddots\ & \vdots & \vdots\\
			y_{L+R,1}-\frac{y_{1,1}c_{-L+1}}{c_{R}} & \cdots & y_{L+R,d}-\frac{y_{1,d}c_{-L+1}}{c_{R}} & c_{-L+1}\\
			-\frac{y_{1,1}c_{-L}}{c_{R}} & \cdots & -\frac{y_{1,d}c_{-L}}{c_{R}} & c_{-L}\\
		\end{bmatrix}
	\end{equation}
	
	At this point, we observe that the non-zero part of the matrix on the left has shifted down by one row, and the two submatrices before and after the transformation satisfy the following equation:
	
	\begin{equation}
		\label{matrixGauss}
		\begin{bmatrix}
			y_{2,1}-\frac{y_{1,1}c_{R-1}}{c_{R}}  & \cdots & y_{2,d}-\frac{y_{1,d}c_{R-1}}{c_{R}}\\
			\vdots & \ddots\ & \vdots\\
			y_{L+R,1}-\frac{y_{1,1}c_{-L+1}}{c_{R}} & \cdots & y_{L+R,d}-\frac{y_{1,d}c_{-L+1}}{c_{R}}\\
			-\frac{y_{1,1}c_{-L}}{c_{R}} & \cdots & -\frac{y_{1,d}c_{-L}}{c_{R}}\\
		\end{bmatrix}
		= 
		\begin{bmatrix}
			-\frac{c_{R-1}}{c_{R}} & 1 & 0  & \cdots & 0\\
			-\frac{c_{R-2}}{c_{R}} & 0 & 1  & \cdots & 0\\
			\vdots & \vdots & \ddots & \ddots & \vdots\\
			-\frac{c_{-L+1}}{c_{R}} & 0 & \cdots & 0 & 1\\
			-\frac{c_{-L}}{c_{R}}  & 0 & \cdots & 0 & 0
		\end{bmatrix}
		\begin{bmatrix}
			y_{1,1} & \cdots & y_{1,d}\\
			y_{2,1} & \cdots & y_{2,d}\\
			\vdots & \ddots\ & \vdots\\
			y_{l+r,1} & \cdots & y_{l+R,d}\\
		\end{bmatrix}
	\end{equation}

	The above equation indicates that a series of Gaussian eliminations on identical columns can be represented as the multiplication of a matrix power from the left. This is useful for computing matrices with a Toeplitz structure. For more details please read Cinkir's \cite{cinkir2014fast}.

	We first identify a column with the complete set of parameters. In a standard Toeplitz matrix, this is the $(R+1)$-th column. Then, perform Gaussian elimination on the left part of the matrix using this column. The row containing $c_{R}$ will have only one non-zero element, which we can extract $(-1)^{R}c_{R}$.Additionally, from Eq. \eqref{det}, this Gaussian elimination is equivalent to left-multiplying the matrix $A$ in Eq. \eqref{det} by the matrix $T$. The Gaussian elimination can be performed at most $n-L-R$ times, resulting in $T^{n-L-R}A$. 
\end{proof}

\subsection{Inverse of banded Toeplitz matrix}

In the previous subsection, we established that Gaussian elimination is equivalent to left-multiplying by a shifting matrix. This method is also useful in computing the inverse of banded Toeplitz matrices. Since $M^{-1} = \frac{M^*}{|M|}$, and the method described works for all matrices with a banded Toeplitz structure, we can quickly compute a portion of the algebraic cofactors, thereby calculating elements of the inverse matrix. We first compute the area in the upper-left $R*L$ region of the matrix $M^{-1}$. 

\begin{theorem}
	Let $x_{i,j}$ represents the element in the $i$-th row and $j$-th column of $M^{-1}$, for $1 \le i \le L, 1 \le j \le R$, $x_{i,j}$ satisfies the following formula:
	
	\begin{equation}
		\label{invformula}
		x_{i,j}=\frac{(-1)^{i+j+r(n-L-2R)+(L-1)(n-L)}c_{R}^{n-L-2R}}{|M|}	
		\begin{vmatrix}
			T^{n-L-2R}A_{i} & B\\
			C_{j,i} & 0
		\end{vmatrix},
	\end{equation}
	\begin{equation}
		\setlength{\arraycolsep}{3.5pt}
		\nonumber
		A=
		\begin{bmatrix}
			c_{-L} & c_{-L+1} & \cdots & c_{R-1}\\
			0 & c_{-L} & \cdots & c_{R-2}\\
			\vdots & \vdots & \ddots & \vdots\\
			0 & 0 & \cdots & c_{-L}
		\end{bmatrix},
		B=
		\begin{bmatrix}
			c_{R} & \cdots & 0\\
			\vdots & \ddots & \vdots \\
			c_{0} & \cdots & c_{R}\\
			\vdots & \ddots & \vdots\\
			c_{-L+1} & \cdots & c_{0}
		\end{bmatrix},
		C=
		\begin{bmatrix}
			c_{0} & \cdots & c_{r} & \cdots & 0\\
			\vdots & \ddots & \vdots & \ddots & \vdots\\
			c_{-L+1} & \cdots & c_{0} & \cdots & c_{R}\\
		\end{bmatrix},
		T=
		\begin{bmatrix}
			-\frac{c_{R-1}}{c_{R}} & 1 & 0  & \cdots & 0\\
			-\frac{c_{R-2}}{c_{R}} & 0 & 1  & \cdots & 0\\
			\vdots & \vdots & \ddots & \ddots & \vdots\\
			-\frac{c_{-L+1}}{c_{R}} & 0 & \cdots & 0 & 1\\
			-\frac{c_{-L}}{c_{R}}  & 0 & \cdots & 0 & 0
		\end{bmatrix}.
	\end{equation}
	$A_i$ represents matrix $A$ with the $i$-th column removed, and $C_{j,i}$ represents the matrix $C$ with both the $j$-th row and $i$-th column removed.
\end{theorem}

\begin{proof}
	
	The algebraic cofactor $cof_{i,j}$ of the element in the $i$-th row and $j$-th column is given by the following formula:
	\begin{equation}
		\label{cof}
		cof_{j,i}=(-1)^{i+j+(L-1)*(n-L)}
		\left|\begin{array}{ccc:ccccc}
			& & & c_{R} &  & \ddots &  & \\
			& A_{i} & & \vdots & \ddots &  & 0 & \\
			& & & c_{0} &  & c_{R} &  & \ddots\\
			& & & \vdots & \ddots & \vdots & \ddots & \\
			& & & c_{-L} &  & c_{0} &  & c_{R}\\
			& 0 & &  & \ddots & \vdots & \ddots & \vdots\\
			& & & 0 &  & c_{-L} & \cdots & c_{0}\\
			\hdashline
			& & & & & & & \\
			& C_{j,i} & & & &0& & \\
			& & & & & & & \\
		\end{array}\right|.
	\end{equation}
	To facilitate understanding, we move the top $L-1$ rows (excluding one row due to the deletion of the algebraic cofactor) of the cofactor matrix to the bottom. Then, we can find the complete $L+1+R$ parameters in the $R$-th column and perform Gaussian elimination on the left part, which is equivalently represented by left-multiplying by the matrix $T$ in Eq. \eqref{invformula}. The Gaussian elimination with complete parameters can be executed at most $n-L-2R$ times. Additionally, it is important to note that the position of the inverse element corresponding to the cofactor is symmetric with respect to the diagonal, so in Eq. \eqref{cof}, it is $cof_{j,i}$ rather than $cof_{i,j}$.
\end{proof}

In the above, we have introduced the method for calculating some of the elements of the inverse matrix. However, calculating the cofactors for all elements still incurs significant computational cost. Therefore, we use the method of solving linear equations to compute the entire inverse matrix.

\begin{theorem}
	Let $x_{i,j}$ represents the element in the $i$-th row and $j$-th column of $M^{-1}$ ($x_{i,j}=0$ if $j \le 0$). Calculating by rows, for $j > L$, $x_{i,j}$ satisfies the following formula:
	
	\begin{equation}
		\label{relationRow}
		x_{i,j} =
		\begin{cases}
			-\frac{1}{c_{-L}}\sum_{t=-L+1}^{R}c_{t}x_{i,j-t+L}    & (i \neq j - L) \\
			\frac{1}{c_{-L}}(1-\sum_{t=-L+1}^{R}c_{t}x_{i,j-t+L}) & (i = j- L)
		\end{cases}
	\end{equation}
\end{theorem}

\begin{proof}
	
	Select $i$-th row $X_i$ ($1\le i \le n$) of the $M^{-1}$. We have $X_iM=E_{i}$, where $E_i$ is the $i$-th row of the identity matrix. Thus, we can set up the following system of equations:

\begin{equation}
	\label{eqRow}
	\begin{cases}
		c_{R}x_{i,j-L-R}+ \cdots +c_{-L}x_{i,j} &= 0 \ \ \ \ \text{if}\ i \neq j - L \\
		c_{R}x_{i,j-L-R}+ \cdots +c_{-L}x_{i,j} &= 1 \ \ \ \ \text{if}\ i = j - L \\
	\end{cases}.
\end{equation}
Based on the above system of equations, we can obtain Eq. \eqref{relationRow}. 
\end{proof}

\begin{theorem}
	Let $x_{i,j}$ represents the element in the $i$-th row and $j$-th column of $M^{-1}$ ($x_{i,j}=0$ if $i \le 0$). Calculating by columns, for $i > R$, $x_{i,j}$ satisfies the following formula:
	\begin{equation}
		\label{relationCol}
		x_{i,j}=
		\begin{cases}
			-\frac{1}{c_{R}}\sum_{t=-L}^{R-1} c_{t}x_{i+t-r,j}    & (i \neq j + R)\\
			\frac{1}{c_{R}}(1-\sum_{t=-L}^{R-1} c_{t}x_{i+t-r,j}) & (i = j + R)
		\end{cases}
	\end{equation}
\end{theorem}

\begin{proof}
	Similarly, we can derive Eq. \eqref{relationCol} by solving the following system of equations for the columns.
\end{proof}
\begin{equation}
	\label{eqCol}
	\begin{cases}
		c_{-L}x_{i-L-R,j}+ \cdots +c_{R}x_{i,j} &= 0 \ \ \ \ \text{if}\ i \neq j + R \\
		c_{-L}x_{i-L-R,j}+ \cdots +c_{R}x_{i,j} &= 1 \ \ \ \ \text{if}\ i = j + R \\
	\end{cases}.
\end{equation}

\begin{corollary}
	The inverse of $n$th-order Toeplitz matrix with bandwidth $k=L+1+R$ can be calculated in $O(k^3 \log n+k^5)+kn^2$.
\end{corollary}
\begin{proof}
	The computation of the inverse of a banded Toeplitz matrix is divided into four steps:
	
	\begin{enumerate}
		\item Compute $T^{n-L-2R}$ and $c_R^{n-L-2R}$, which is efficiently done using simple method of bisection, with a complexity of $O(k^3\log n)$.
		\item Calculate $|M|$ based on Eq. \eqref{det}, with a complexity of $O(k^3)$.
		\item For $1 \le i \le R,1\le j \le L$, calculate $x_{i,j}$ based on Eq. \eqref{invformula}, with a complexity of $O(k^5)$ and parallel computing available.
		\item For $1 \le i \le R$, $L <j \le n$, calculate $x_{i,j}$ based on Eq. \eqref{relationRow}. For $R < i \le n, 1 \le j \le n$, calculate $x_{i,j}$ based on Eq. \eqref{relationCol}. The time cost of this part is $kn^2$ and parallel computing is available.
	\end{enumerate}

	The pseudocode of this algorithm is shown in Algorithm \ref{a1}.
	\begin{algorithm}[h]
		\SetAlgoLined
		\label{a1}
		Input a banded Toeplitz matrix with parameters $c_{-L},\cdots,c_{R}$ and order $n$\;
		calculate $T^{n-L-2R}$ and $c_k^{n-L-2R}$ \tcp*{\textbf{step1}}
		calculate $|M|$	\tcp*{\textbf{step2}}
		\For{$i = 1; i \leq R; i++$}{
			\For{$j = 1;  i \leq L; j++$}{
				calculate $x_{i,j}$ according to Eq. \eqref{invformula} \tcp*{\textbf{step3}}
			}
		}
		\For{$i = 1; i \leq R; i++$}{
			\For{$j = L+1; j \leq n; j++$}{
				calculate $x_{i,j}$ according to left of Eq. \eqref{relationRow} \tcp*{\textbf{step4}}
			}	
		}
		\For{$i = R+1; i \leq n; i++$}{
			\For{$j = 1;j \leq n; j++$}{
				calculate $x_{i,j}$ according to right of Eq. \eqref{relationCol}\;
			}
			
		}
		construct and return $M^{-1}$
		\caption{Inverse of a $n$th-order banded Toeplitz matrix $M$.}
	\end{algorithm}
\end{proof}

\section{Periodic determinant and inverse of banded Toeplitz matrix over $\mathbb{F}_p$ \label{s3}}

In the previous section, we found that the inverse of a $k$-diagonal Toeplitz matrix can be represented by multiple recurrence relations. Furthermore, over finite fields, these recurrence relations bring about a very favorable property: periodicity, which will greatly simplify calculations over finite fields. Next, we will introduce a well-known concept in finite fields: the linear feedback shift register (LFSR).

\begin{definition}
	LFSR consists of a register and a feedback polynomial. The register used in this paper contains $L+R$ storage units, and the feedback polynomial is usually represented in polynomial form as 
	\begin{equation}
		\label{poly}
		f(x)=c_{r}x^{L+R}+\cdots+c_{-L+1}x+c_{-L} \ \ \  (c_{-L},c_{R} \ne 0).
	\end{equation}

	The working process of an LFSR can be divided into the following steps:
	\begin{description}
		\item[Step1:] Set the register to the initial state (seed).
		\item[Step2:] Shift the numbers in the register one position to the right.
		\item[Step3:] Calculate the new input number based on the feedback polynomial and place it in the leftmost position of the register.
	\end{description}
	Running the LFSR will result in the following recurrence relation:
	\begin{equation}
		\label{recur}
		x_i = -(x_{i-1}c_{R-1}+x_{i-2}c_{R-2}+\cdots+x_{i-L-R}c_{-L})/c_{R}
	\end{equation}

\end{definition}

\begin{proposition}
	\label{period}
	\cite{berlekamp2015algebraic} The period of the LFSR is equal to the period of its corresponding feedback polynomial $f$. The period $P(f)$ of polynomial is $P(f) = \min \{q \in \mathbb{Z}^+\ : \  f(x)\ |\ (x^{q}-1)\}$.
\end{proposition}

\begin{proposition}
	\label{scope}
	\cite{berlekamp2015algebraic} For $f(x)=c_{r}x^{L+R}+\cdots+c_{-L+1}x+c_{-L}$ $(c_{-L},c_{R} \ne 0)$ over $\mathbb{F}_p$, $L+R \le  P(f) \le p^{L+R}-1$.
\end{proposition}

\begin{proposition}
	\label{reverse}
	\cite{berlekamp2015algebraic} For $f(x)=c_{r}x^{L+R}+\cdots+c_{-L+1}x+c_{-L}$ $(c_{-L},c_{R} \ne 0)$ and its reverse polynomial $g(x)=c_{-L}x^{L+R}\cdots+c_{R-1}x+c_R$ over $\mathbb{F}_p$, $P(f)$ = $P(g)$.
\end{proposition}

The calculation of the period of a general polynomial consists of the following three parts. And the complexity is $O(k)^3$.

\begin{enumerate}
	\item Factorize $f$. If $f$ is irreducible, proceed directly to the second step and skip the third step. Numerous factorization algorithms have been studied. 
	\item Calculate the period $P(f_i)$ for each irreducible factor $f_i$. This process is detailed on pages 150–153 of "Algebraic Coding Theory" \cite{berlekamp2015algebraic}.
	\item Calculate $P(f)$ according Theorem \ref{pf}. Fig. \ref{LFSR} shows the overall process of calculating $P(f)$.
\end{enumerate}

\begin{proposition}
	\label{pf}
	 Let $f$ be a reducible polynomial and $f(0) \ne 0$. Assume $f = f_1^{e_1} f_2^{e_2} \cdots f_r^{e_r}$ , where $f_1^{e_1},f_2^{e_2}, \cdots, f_r^{e_r}$ are $r$ different irreducible polynomials over $\mathbb{F}_p$, and $e_1,e_2,\cdots,e_r$ are $r$ positive integers, then we get:
	 \[
		P(f) = \text{lcm}[P(f_1),P(f_2),\cdots,P(f_r)]*\min\{p^t|p^t \ge e_1,e_2,\cdots,e_r\},	 
	 \]
	  where ``lcm” means the least common multiple.
\end{proposition}

\begin{figure}[h]
	\center
	\includegraphics[width=1\linewidth]{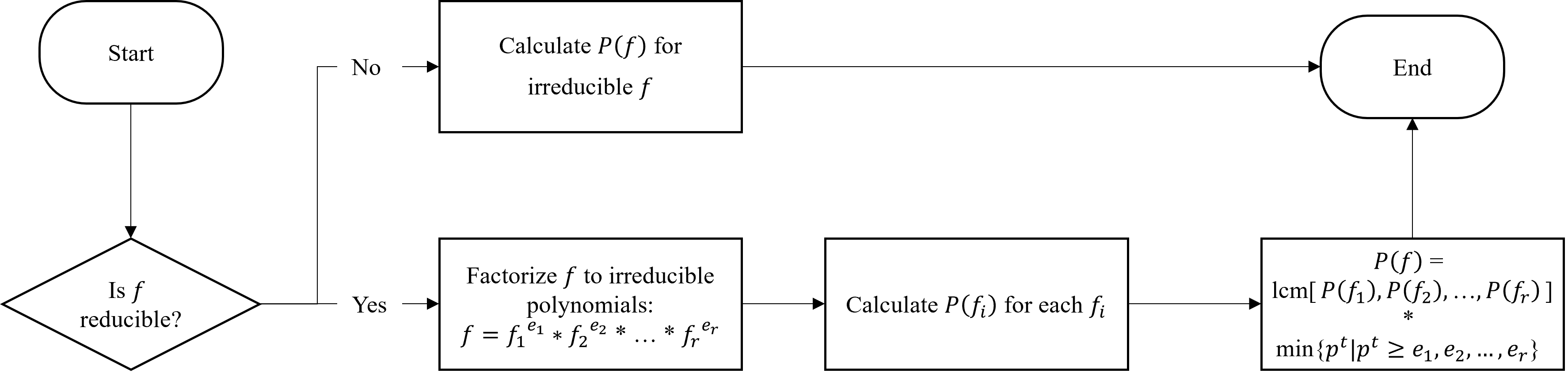}
	\caption{the process of calculating the period of $f$}
	\label{LFSR}
\end{figure}

The above proposition and theorem briefly introduce the period of LFSR and its calculation method. Next, we will relate LFSR to the inverse of banded Toeplitz matrices over finite fields.

\begin{theorem}
	\label{periodT}
	The period of $T$ in Eq. \eqref{invformula} is equal to the period $P(f)$ of the LFSR with feedback polynomial  $f(x)=c_{r}x^{L+R}+\cdots+c_{-L+1}x+c_{-L}$ $(c_{-L},c_{R} \ne 0)$. That is, $T^{P(f)}=E$ where $P(f) = \min \{q \in \mathbb{Z}_+\ : \  f(x)\ |\ (x^{q}-1)\}$
\end{theorem}
\begin{proof}
	Consider the recurrence relation in Eq. \eqref{recur}, where its transition matrix is exactly $T$, as shown in Eq. \eqref{transform}.
	\begin{equation}
		\label{transform}
		\begin{bmatrix}
			x_{i} \\x_{i-1} \\x_{i-2} \\\vdots  \\x_{i-L-R+1}
		\end{bmatrix}' = 
		\begin{bmatrix}
			x_{i-1} \\x_{i-2} \\x_{i-3} \\\vdots  \\x_{i-L-R}
		\end{bmatrix}'
		T=
		\begin{bmatrix}
			-\frac{c_{R-1}}{c_{R}} & 1 & 0  & \cdots & 0\\
			-\frac{c_{R-2}}{c_{R}} & 0 & 1  & \cdots & 0\\
			\vdots & \vdots & \ddots & \ddots & \vdots\\
			-\frac{c_{-L+1}}{c_{R}} & 0 & \cdots & 0 & 1\\
			-\frac{c_{-L}}{c_{R}}  & 0 & \cdots & 0 & 0
		\end{bmatrix}
	\end{equation}
	Given that the period of $f(x)$ is $P(f)$, $ \forall i \in \mathbb{Z}_+$, we have $x_i=x_{i+P(f)}$. Then we can derive the following equation:
	
	\begin{equation}
		\label{transP}
		\begin{bmatrix}
			x_{i-1} \\x_{i-2} \\x_{i-3} \\\vdots  \\x_{i-L-R}
		\end{bmatrix}' =
		\begin{bmatrix}
			x_{i-1+P(f)} \\x_{i-2+P(f)} \\x_{i-3+P(f)} \\\vdots  \\x_{i-L-R+P(f)}
		\end{bmatrix}' =     
		\begin{bmatrix}
			x_{i-1} \\x_{i-2} \\x_{i-3} \\\vdots  \\x_{i-L-R}
		\end{bmatrix}'
		\begin{bmatrix}
			-\frac{c_{R-1}}{c_{R}} & 1 & 0  & \cdots & 0\\
			-\frac{c_{R-2}}{c_{R}} & 0 & 1  & \cdots & 0\\
			\vdots & \vdots & \ddots & \ddots & \vdots\\
			-\frac{c_{-L+1}}{c_{R}} & 0 & \cdots & 0 & 1\\
			-\frac{c_{-L}}{c_{R}}  & 0 & \cdots & 0 & 0
		\end{bmatrix}^{P(f)}.
	\end{equation}
	For any $i$, Eq. \eqref{transP} always holds, so $T^{P(f)}=E$. Since $P(f)$ is the minimum period of $f$, it is also the minimum value for which $T^{P(f)}=E$.
\end{proof}

\begin{corollary}
	The period of determinant of $n$-order Toeplitz matrix with bandwidth $k=L+1+R$ over $\mathbb{F}_p$ is lcm$[p-1,P(f)]$ and can be calculated in $O(k^4)$.
\end{corollary}
\begin{proof}
	According to Eq. \eqref{det} and the periodicity of matrix $T$, we only need to compute $T^{n-L-R \mod P(f)}$ instead of $T^{n-L-R}$ over $\mathbb{F}_p$. The complexity of computing $T^{n-L-R \mod P(f)}$ using the method of bisection is $O(k^3\log P(f))$. Furthermore, based on Proposition \ref{scope}, we have $P(f) \le p^{L+R}-1$, so $\log P(f) \approx  k$ and the final complexity is $O(k^4)$. Over finite fields we have $c_k^{p-1}=1$ and $T^{P(f)}=E$, so the period of the determinant is lcm$[p-1,P(f)]$.
\end{proof}

Now, calculate the inverse. Starting from the top-left of $M^{-1}$, we partition the matrix into blocks of
size $P(f)*P(f)$. This results in $\lceil n/P(f) \rceil^2$ blocks, which are represented $B_{i,j}$. If $P(f) \nmid n$, the last block in each row and column will be smaller than $P(f)*P(f)$.

\begin{theorem}
	\label{row}
	For $i-2 \leq j \leq \lfloor n/P(f) \rfloor$, $B_{i,j} = B_{i,j-1}$.
\end{theorem}
\begin{proof}
	
	All elements in these submatrices are at least $P(f)+1$ away from the diagonal. When computing the value of each element in $B_{i,j}$ by rows, only the recursive formula in upper of Eq. \eqref{relationRow} is applied because $P(f)+1 > L$. Let $\hat{B_{i,j}}$ be the arbitrary consecutive $L+R$ columns in $B_{i,j}$ and we have:
	\begin{equation}
		\hat{B_{i,j}}=\hat{B_{i,j-1}}\hat{T}^{P(F)}=\hat{B_{i,j-1}}
		\begin{bmatrix}
			0 & 0 & \cdots & 0 & -\frac{c_{R}}{c_{-L}}\\
			1 & 0 & \cdots & 0 & -\frac{c_{R-1}}{c_{-L}}\\
			\vdots & \ddots  & \ddots & \vdots & \vdots\\
			0 & \cdots & 1 & 0 & -\frac{c_{-L+2}}{c_{-L}}\\
			0 & \cdots & 0 & 1 & -\frac{c_{-L+1}}{c_{-L}}
		\end{bmatrix}^{P(f)}
	\end{equation}
	
	We can derive $\hat{T}^{P(F)}=E$ using Proposition \ref{period}, \ref{reverse} and Theorem \eqref{periodT}. Then $\hat{B_{i,j}}=\hat{B_{i,j-1}}$. Any consecutive $L+R$ columns of $B_{i,j}$ and $B_{i,j-1}$ are equal, so $B_{i,j} = B_{i,j-1}$.
\end{proof}

\begin{theorem}
	\label{col}
	For $i \leq j-2$ and $i \leq \lfloor n/P(f) \rfloor-1$, $B_{i,j} = B_{i+1,j}$.
\end{theorem}
\begin{proof}
	
	The proof of this theorem is similar to the previous one. When calculating each element of $M^{-1}$ by column, the recursive formula in upper right of Eq. \eqref{relationCol} is applied because $P(f)+1>R$. Let $\bar{B_{i,j}}$ be the arbitrary consecutive $L+R$ rows in $B_{i,j}$ and we have:
	
	\begin{equation}
	\bar{B_{i+1,j}} = \bar{B_{i,j}}\bar{T}^{P(f)} =  \bar{B_{i,j}}
		\begin{bmatrix}
			0 & 1 & \cdots & 0 & 0\\
			0 & 0 & \ddots & \vdots & \vdots\\
			\vdots & \vdots & \ddots & 1 & 0\\
			0 & 0 & \cdots & 0 & 1\\
			-\frac{c_{-L}}{c_{R}} & -\frac{c_{-L+1}}{c_{R}} & \cdots & -\frac{c_{R-2}}{c_{R}} & -\frac{c_{R-1}}{c_{R}}
		\end{bmatrix}^{P(f)} 
	\end{equation}
	
	$\bar{T}$ is the antidiagonal transpose of $T$ in Theorem \eqref{periodT} and has the same period, so $\bar{T}^{P(f)}=E$. Then $\bar{B_{i,j}}=\bar{B_{i+1,j}}$. Any consecutive $L+R$ rows of $B_{i,j}$ and $B_{i+1,j}$ are equal, so $B_{i,j} = B_{i+1,j}$.
\end{proof}

\begin{theorem}
	\label{up}
	For $1 \le i < j \le \lfloor n/P(f) \rfloor$, $B_{i,j}$ = $B_{1,2}$, and $B_{i,\lceil n/P(f) \rceil}$ is equal to the left $n \mod P(f)$ columns of $B_{1,2}$.
\end{theorem}
\begin{proof}
	The proof of this theorem can be simply derived from Theorem \eqref{row} and \eqref{col}. Furthermore, the proof of $B_{i,\lceil n/P(f) \rceil}$ is similar to the proofs of the previous two theorems, and thus is omitted here.
\end{proof}

\begin{theorem}
	\label{down}
	For $1 \le j < i \le \lfloor n/P(f) \rfloor$, $B_{i,j}$ = $B_{2,1}$, and $B_{\lceil n/P(f) \rceil,j}$ is equal to the upper $n \mod P(f)$ rows of $B_{2,1}$.
\end{theorem}
\begin{proof}	
	We have previously proven that all block matrices above the diagonal are equal, and the same holds for those below the diagonal. The proof is not only similar, but we can also obtain this conclusion by transposing the original matrix (which converts the blocks below the diagonal to those above the diagonal).
\end{proof}

\begin{theorem}
	\label{mid}
	For $1 \le i \le \lfloor n/P(f) \rfloor$, $B_{i,i}$ = $B_{1,1}$, and $B_{\lceil n/P(f) \rceil,\lceil n/P(f) \rceil}$ is equal to the upper left $(n \mod P(f))*(n \mod P(f))$ submatrix of $B_{1,1}$.
\end{theorem}
\begin{proof}
	From Theorem 1, we have proven that for any $i$, $B_{i,i+1}=B_{1,2}$. According to Eq. \eqref{relationRow}  , $B_{i,i}$ and $B_{1,1}$ can be obtained from $B_{i,i+1}$ and $B_{1,2}$ through the same transformation, therefore $B_{i,i}$ = $B_{1,1}$. For the bottom-right submatrix, because of its smaller size, it is equal to the top-left part of the submatrices on the diagonal.
\end{proof}

\begin{corollary}
	The inverse of $n$th-order $k$-diagonal Toeplitz matrix over $\mathbb{F}_p$ can be calculated in $O(k^5)+3kP(f)^2$.
\end{corollary}
\begin{proof}
	From Theorem \ref{up}, \ref{down} and \ref{mid}, we know that the inverse of a Toeplitz matrix over the field $\mathbb{F}_p$ has a periodic structure, and the period is $P(f)$. It can be represented by three $P(f)*P(f)$ matrices located on the diagonal, above the diagonal, and below the diagonal. Therefore, in the step 4 of Algorithm \ref{a1}, our computational cost is reduced from $kn^2$
	to $3kP(f)^2$. Theorem \ref{periodT} proves that $T^{P(f)}=E$, all calculations involving powers of $T$ can be taken modulo $P(f)$, and $\log P(f) \approx k$, thus reducing the overall complexity to $O(k^5)+3kP(f)^2$.

\end{proof}

\begin{example}
	Consider a 26-order Toeplitz matrix with parameters $[c_{-2},c_{-1},c_0,c_1,c_2]=[1,1,1,1,1]$ over $\mathbb{F}_2$. We can calculate that its inverse period is 5 because $x^5+1=(x + 1)*(x^4 + x^3 + x^2 + x + 1)$. Its inverse is shown in Fig. \ref{ex} left.
\end{example}

\begin{example}
	Consider the 26-order Toeplitz matrices with parameters $[c_{-2},c_{-1},c_0,c_1,c_2]=[1,1,0,1,1]$ over $\mathbb{F}_2$. First, perform factorization: $x^4+x^3+x+1=(x+1)^2(x^2+x+1)$. Next, calculate the period of each factor, which is 3 because $x^3+1=(x+1)(x^2+x+1)$. Finally, calculate the total period $P(f) = \text{lcm}[3,3]*\min\{2^t|2^t \ge 1,2\} = 6$. Its inverse is shown in Fig. \ref{ex} right.
\end{example}
\begin{figure}[h]
	\center
	\includegraphics[width=0.3\linewidth]{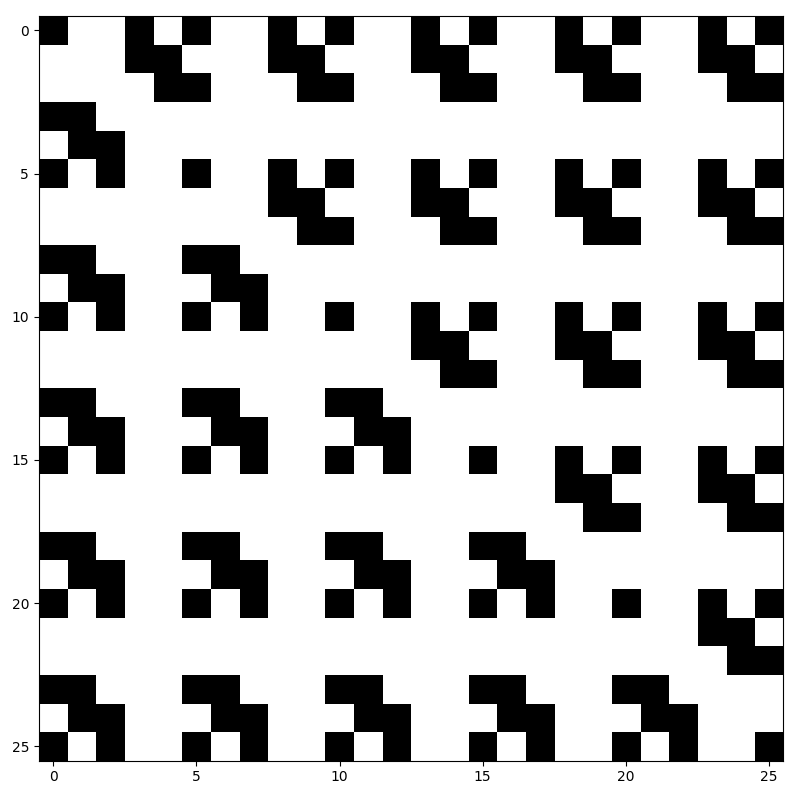}
	\includegraphics[width=0.3\linewidth]{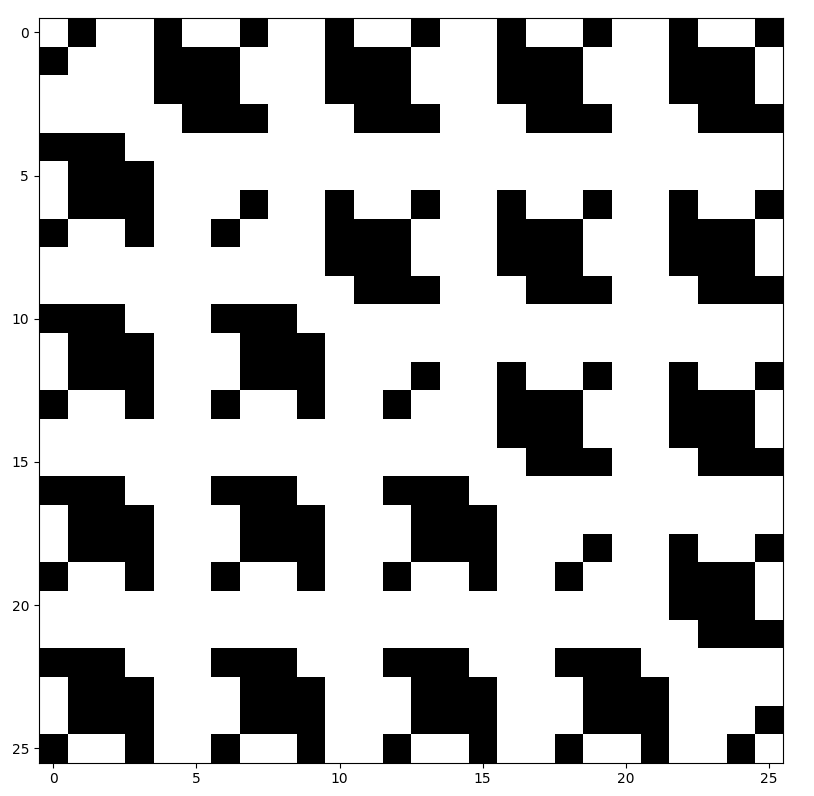}
	
	\caption{The inverse of two 26-order Toeplitz matrix with parameters [1,1,1,1,1] and [1,1,0,1,1]}
	\label{ex}
\end{figure}
We can find that the the inverses over $\mathbb{F}_2$ are periodic. After computing the three small matrices respectively located on the diagonal, above the diagonal, and below the diagonal, the entire inverse matrix can be derived without additional calculations.
	\begin{algorithm}[h]
	\SetAlgoLined
	\label{a2}
	input a banded Toeplitz matrix with parameters $c_{-L},\cdots,c_{R}$ and order $n$\;
	factorize $f$\;
	calculate the period $P(f_i)$ for each irreducible factor $f_i$.
	calculate $P(f)$ according Theorem \ref{pf}\;
	calculate $T^{n-L-2R \mod P(f)}$ and $c_k^{n-L-2R \mod P(f)}$\;
	\For {$1 \le i \le R,1\le j \le L$} {
		calculate $x_{i,j}$ based on Eq. \eqref{invformula}\;
	}
	\For {$1 \le i \le R$, $L <j \le 2P(f)$}{
		calculate $x_{i,j}$ based on Eq. \eqref{relationRow}\;
	}
	
	\For{$R+1 \le i \le P(f)$, $1 <j \le 2P(f)$}{
		calculate $x_{i,j}$ based on Eq. \eqref{relationCol}\;
	}
	\For{$P(f)+1 \le i \le 2P(f)$, $1 <j \le P(f)$}{
		calculate $x_{i,j}$ based on Eq. \eqref{relationCol}\;
	}
	construct and return $M^{-1}$
	\caption{Inverse of a $n$th-order banded Toeplitz matrix $M$.}
\end{algorithm}

\newpage

\section{Conclusion\label{s4}}
We have identified the minimal periods for the determinant and the inverse of a Toeplitz matrix with bandwidth $k$ over $\mathbb{F}_p$ and provide detailed calculations and proofs. For the determinant, the computation complexity is $O(k^4)$, which is independent of the matrix size. This implies that we can
 perform efficient computations for almost any matrix. For the inverse, the computation complexity is $O(k^5)+3kP(f)^2$, where $P(f) \le p^{k-1}-1$. This means that even in the worst cases, when $P(f)$ reaches maximum, we are still able to compute the inverse for high-order Toeplitz matrices with bandwidth less than 20 over $\mathbb{F}_2$.


\section*{Acknowledgments}
This study is financed by Tianjin Science and Technology Bureau, finance code: 21JCYBJC00210.





\bibliographystyle{elsarticle-harv}
\bibliography{ref}

\begin{thebibliography}{31}
\expandafter\ifx\csname natexlab\endcsname\relax\def\natexlab#1{#1}\fi
\providecommand{\url}[1]{\texttt{#1}}
\providecommand{\href}[2]{#2}
\providecommand{\path}[1]{#1}
\providecommand{\DOIprefix}{doi:}
\providecommand{\ArXivprefix}{arXiv:}
\providecommand{\URLprefix}{URL: }
\providecommand{\Pubmedprefix}{pmid:}
\providecommand{\doi}[1]{\href{http://dx.doi.org/#1}{\path{#1}}}
\providecommand{\Pubmed}[1]{\href{pmid:#1}{\path{#1}}}
\providecommand{\bibinfo}[2]{#2}
\ifx\xfnm\relax \def\xfnm[#1]{\unskip,\space#1}\fi
\bibitem[{Aceto et~al.(2012a)Aceto, Ghelardoni and Magherini}]{ACETO20122960}
\bibinfo{author}{Aceto, L.}, \bibinfo{author}{Ghelardoni, P.},
  \bibinfo{author}{Magherini, C.}, \bibinfo{year}{2012}a.
\newblock \bibinfo{title}{Boundary value methods for the reconstruction of
  sturm–liouville potentials}.
\newblock \bibinfo{journal}{Applied Mathematics and Computation}
  \bibinfo{volume}{219}, \bibinfo{pages}{2960--2974}.
\newblock \URLprefix
  \url{https://www.sciencedirect.com/science/article/pii/S0096300312009204},
  \DOIprefix\doi{https://doi.org/10.1016/j.amc.2012.09.021}.
\bibitem[{Aceto et~al.(2012b)Aceto, Ghelardoni and Magherini}]{ACETO20123857}
\bibinfo{author}{Aceto, L.}, \bibinfo{author}{Ghelardoni, P.},
  \bibinfo{author}{Magherini, C.}, \bibinfo{year}{2012}b.
\newblock \bibinfo{title}{Pgscm: A family of p-stable boundary value methods
  for second-order initial value problems}.
\newblock \bibinfo{journal}{Journal of Computational and Applied Mathematics}
  \bibinfo{volume}{236}, \bibinfo{pages}{3857--3868}.
\newblock \URLprefix
  \url{https://www.sciencedirect.com/science/article/pii/S0377042712001525},
  \DOIprefix\doi{https://doi.org/10.1016/j.cam.2012.03.024}. \bibinfo{note}{40
  years of numerical analysis: “Is the discrete world an approximation of the
  continuous one or is it the other way around?”}.
\bibitem[{Akin(2021)}]{Akin2021}
\bibinfo{author}{Akin, H.}, \bibinfo{year}{2021}.
\newblock \bibinfo{title}{Description of reversibility of 9-cyclic 1d finite
  linear cellular automata with periodic boundary conditions}.
\newblock \bibinfo{journal}{Journal of Cellular Automata} \bibinfo{volume}{16},
  \bibinfo{pages}{127–151}.
\bibitem[{Berlekamp(2015)}]{berlekamp2015algebraic}
\bibinfo{author}{Berlekamp, E.R.}, \bibinfo{year}{2015}.
\newblock \bibinfo{title}{Algebraic coding theory (revised edition)}.
\newblock \bibinfo{publisher}{World Scientific}.
\bibitem[{Bini and Meini(1999)}]{bini1999effective}
\bibinfo{author}{Bini, D.A.}, \bibinfo{author}{Meini, B.},
  \bibinfo{year}{1999}.
\newblock \bibinfo{title}{Effective methods for solving banded toeplitz
  systems}.
\newblock \bibinfo{journal}{SIAM Journal on Matrix Analysis and Applications}
  \bibinfo{volume}{20}, \bibinfo{pages}{700--719}.
\bibitem[{Cinkir(2012)}]{CINKIR20122298}
\bibinfo{author}{Cinkir, Z.}, \bibinfo{year}{2012}.
\newblock \bibinfo{title}{An elementary algorithm for computing the determinant
  of pentadiagonal toeplitz matrices}.
\newblock \bibinfo{journal}{Journal of Computational and Applied Mathematics}
  \bibinfo{volume}{236}, \bibinfo{pages}{2298--2305}.
\newblock \URLprefix
  \url{https://www.sciencedirect.com/science/article/pii/S0377042711005875},
  \DOIprefix\doi{https://doi.org/10.1016/j.cam.2011.11.017}.
\bibitem[{Cinkir(2014)}]{cinkir2014fast}
\bibinfo{author}{Cinkir, Z.}, \bibinfo{year}{2014}.
\newblock \bibinfo{title}{A fast elementary algorithm for computing the
  determinant of toeplitz matrices}.
\newblock \bibinfo{journal}{Journal of computational and applied mathematics}
  \bibinfo{volume}{255}, \bibinfo{pages}{353--361}.
\bibitem[{Cinkir et~al.(2011)Cinkir, Akin and Siap}]{cinkir2011reversibility}
\bibinfo{author}{Cinkir, Z.}, \bibinfo{author}{Akin, H.},
  \bibinfo{author}{Siap, I.}, \bibinfo{year}{2011}.
\newblock \bibinfo{title}{Reversibility of 1d cellular automata with periodic
  boundary over finite fields}.
\newblock \bibinfo{journal}{Journal of Statistical Physics}
  \bibinfo{volume}{143}, \bibinfo{pages}{807--823}.
\bibitem[{Du et~al.(2022)Du, Wang, Wang and Gao}]{DU2022163}
\bibinfo{author}{Du, X.}, \bibinfo{author}{Wang, C.}, \bibinfo{author}{Wang,
  T.}, \bibinfo{author}{Gao, Z.}, \bibinfo{year}{2022}.
\newblock \bibinfo{title}{Efficient methods with polynomial complexity to
  determine the reversibility of general 1d linear cellular automata over zp}.
\newblock \bibinfo{journal}{Information Sciences} \bibinfo{volume}{594},
  \bibinfo{pages}{163--176}.
\newblock \URLprefix
  \url{https://www.sciencedirect.com/science/article/pii/S0020025522000743},
  \DOIprefix\doi{https://doi.org/10.1016/j.ins.2022.01.045}.
\bibitem[{Dumas and Pernet(2012)}]{dumas2012computational}
\bibinfo{author}{Dumas, J.G.}, \bibinfo{author}{Pernet, C.},
  \bibinfo{year}{2012}.
\newblock \bibinfo{title}{Computational linear algebra over finite fields}.
\newblock \bibinfo{journal}{arXiv preprint arXiv:1204.3735} .
\bibitem[{Gray et~al.(2006)}]{gray2006toeplitz}
\bibinfo{author}{Gray, R.M.}, et~al., \bibinfo{year}{2006}.
\newblock \bibinfo{title}{Toeplitz and circulant matrices: A review}.
\newblock \bibinfo{journal}{Foundations and Trends{\textregistered} in
  Communications and Information Theory} \bibinfo{volume}{2},
  \bibinfo{pages}{155--239}.
\bibitem[{Hadj and Elouafi(2008)}]{hadj2008fast}
\bibinfo{author}{Hadj, A.D.A.}, \bibinfo{author}{Elouafi, M.},
  \bibinfo{year}{2008}.
\newblock \bibinfo{title}{A fast numerical algorithm for the inverse of a
  tridiagonal and pentadiagonal matrix}.
\newblock \bibinfo{journal}{Applied Mathematics and Computation}
  \bibinfo{volume}{202}, \bibinfo{pages}{441--445}.
\bibitem[{Heinig and Rost(2011)}]{heinig2011fast}
\bibinfo{author}{Heinig, G.}, \bibinfo{author}{Rost, K.}, \bibinfo{year}{2011}.
\newblock \bibinfo{title}{Fast algorithms for toeplitz and hankel matrices}.
\newblock \bibinfo{journal}{Linear Algebra and its Applications}
  \bibinfo{volume}{435}, \bibinfo{pages}{1--59}.
\bibitem[{Jia et~al.(2016)Jia, Yang and Li}]{jia2016homogeneous}
\bibinfo{author}{Jia, J.}, \bibinfo{author}{Yang, B.}, \bibinfo{author}{Li,
  S.}, \bibinfo{year}{2016}.
\newblock \bibinfo{title}{On a homogeneous recurrence relation for the
  determinants of general pentadiagonal toeplitz matrices}.
\newblock \bibinfo{journal}{Computers \& Mathematics with Applications}
  \bibinfo{volume}{71}, \bibinfo{pages}{1036--1044}.
\bibitem[{Kilic and El-Mikkawy(2008)}]{kilic2008computational}
\bibinfo{author}{Kilic, E.}, \bibinfo{author}{El-Mikkawy, M.},
  \bibinfo{year}{2008}.
\newblock \bibinfo{title}{A computational algorithm for special nth-order
  pentadiagonal toeplitz determinants}.
\newblock \bibinfo{journal}{Applied mathematics and computation}
  \bibinfo{volume}{199}, \bibinfo{pages}{820--822}.
\bibitem[{Lv and Huang(2007)}]{LV20071189}
\bibinfo{author}{Lv, X.G.}, \bibinfo{author}{Huang, T.Z.},
  \bibinfo{year}{2007}.
\newblock \bibinfo{title}{A note on inversion of toeplitz matrices}.
\newblock \bibinfo{journal}{Applied Mathematics Letters} \bibinfo{volume}{20},
  \bibinfo{pages}{1189--1193}.
\newblock \URLprefix
  \url{https://www.sciencedirect.com/science/article/pii/S0893965907000535},
  \DOIprefix\doi{https://doi.org/10.1016/j.aml.2006.10.008}.
\bibitem[{Lv et~al.(2008)Lv, Huang and Le}]{lv2008note}
\bibinfo{author}{Lv, X.G.}, \bibinfo{author}{Huang, T.Z.}, \bibinfo{author}{Le,
  J.}, \bibinfo{year}{2008}.
\newblock \bibinfo{title}{A note on computing the inverse and the determinant
  of a pentadiagonal toeplitz matrix}.
\newblock \bibinfo{journal}{Applied Mathematics and Computation}
  \bibinfo{volume}{206}, \bibinfo{pages}{327--331}.
\bibitem[{MacWilliams and Sloane(1977)}]{macwilliams1977theory}
\bibinfo{author}{MacWilliams, F.J.}, \bibinfo{author}{Sloane, N.J.A.},
  \bibinfo{year}{1977}.
\newblock \bibinfo{title}{The theory of error-correcting codes}.
  volume~\bibinfo{volume}{16}.
\newblock \bibinfo{publisher}{Elsevier}.
\bibitem[{Meek(1983)}]{meek1983inverses}
\bibinfo{author}{Meek, D.}, \bibinfo{year}{1983}.
\newblock \bibinfo{title}{The inverses of toeplitz band matrices}.
\newblock \bibinfo{journal}{Linear Algebra and its Applications}
  \bibinfo{volume}{49}, \bibinfo{pages}{117--129}.
\bibitem[{Ng et~al.(2002)Ng, Rost and Wen}]{ng2002inversion}
\bibinfo{author}{Ng, M.K.}, \bibinfo{author}{Rost, K.}, \bibinfo{author}{Wen,
  Y.W.}, \bibinfo{year}{2002}.
\newblock \bibinfo{title}{On inversion of toeplitz matrices}.
\newblock \bibinfo{journal}{Linear algebra and its applications}
  \bibinfo{volume}{348}, \bibinfo{pages}{145--151}.
\bibitem[{Poletti and Teal(2021)}]{poletti2021superfast}
\bibinfo{author}{Poletti, M.A.}, \bibinfo{author}{Teal, P.D.},
  \bibinfo{year}{2021}.
\newblock \bibinfo{title}{A superfast toeplitz matrix inversion method for
  single-and multi-channel inverse filters and its application to room
  equalization}.
\newblock \bibinfo{journal}{IEEE/ACM Transactions on Audio, Speech, and
  Language Processing} \bibinfo{volume}{29}, \bibinfo{pages}{3144--3157}.
\bibitem[{Mart{\'\i}n~del Rey and
  Rodr{\'\i}guez~S{\'a}nchez(2015)}]{martin2015reversible}
\bibinfo{author}{Mart{\'\i}n~del Rey, A.},
  \bibinfo{author}{Rodr{\'\i}guez~S{\'a}nchez, G.}, \bibinfo{year}{2015}.
\newblock \bibinfo{title}{Reversible elementary cellular automaton with rule
  number 150 and periodic boundary conditions over fp}.
\newblock \bibinfo{journal}{International Journal of Modern Physics C}
  \bibinfo{volume}{26}, \bibinfo{pages}{1550120}.
\bibitem[{del Rey(2015)}]{del2015note}
\bibinfo{author}{del Rey, A.M.}, \bibinfo{year}{2015}.
\newblock \bibinfo{title}{A note on the reversibility of the elementary
  cellular automaton with rule number 90}.
\newblock \bibinfo{journal}{Rev. Un. Mat. Argentina} \bibinfo{volume}{56}.
\bibitem[{del Rey et~al.(2011)del Rey, Rodr{\i}
  et~al.}]{marti2011reversibility}
\bibinfo{author}{del Rey, {\'A}.M.}, \bibinfo{author}{Rodr{\i}, G.}, et~al.,
  \bibinfo{year}{2011}.
\newblock \bibinfo{title}{Reversibility of linear cellular automata}.
\newblock \bibinfo{journal}{Applied Mathematics and Computation}
  \bibinfo{volume}{217}, \bibinfo{pages}{8360--8366}.
\bibitem[{del Rey and S{\'a}nchez(2013)}]{del2013invertible}
\bibinfo{author}{del Rey, {\'A}.M.}, \bibinfo{author}{S{\'a}nchez, G.R.},
  \bibinfo{year}{2013}.
\newblock \bibinfo{title}{On the invertible cellular automata 150 over fp}.
\newblock \bibinfo{journal}{Applied Mathematics and Computation}
  \bibinfo{volume}{219}, \bibinfo{pages}{5427--5432}.
\bibitem[{Silverman et~al.(2008)Silverman, Pipher and
  Hoffstein}]{silverman2008introduction}
\bibinfo{author}{Silverman, J.H.}, \bibinfo{author}{Pipher, J.},
  \bibinfo{author}{Hoffstein, J.}, \bibinfo{year}{2008}.
\newblock \bibinfo{title}{An introduction to mathematical cryptography}.
  volume~\bibinfo{volume}{1}.
\newblock \bibinfo{publisher}{Springer}.
\bibitem[{Trench(1974)}]{trench1974inversion}
\bibinfo{author}{Trench, W.F.}, \bibinfo{year}{1974}.
\newblock \bibinfo{title}{Inversion of toeplitz band matrices}.
\newblock \bibinfo{journal}{Mathematics of computation} \bibinfo{volume}{28},
  \bibinfo{pages}{1089--1095}.
\bibitem[{Trench(1985)}]{trench1985explicit}
\bibinfo{author}{Trench, W.F.}, \bibinfo{year}{1985}.
\newblock \bibinfo{title}{Explicit inversion formulas for toeplitz band
  matrices}.
\newblock \bibinfo{journal}{SIAM Journal on Algebraic Discrete Methods}
  \bibinfo{volume}{6}, \bibinfo{pages}{546--554}.
\bibitem[{Wang et~al.(2015)Wang, Li and Zhao}]{wang2015explicit}
\bibinfo{author}{Wang, C.}, \bibinfo{author}{Li, H.}, \bibinfo{author}{Zhao,
  D.}, \bibinfo{year}{2015}.
\newblock \bibinfo{title}{An explicit formula for the inverse of a
  pentadiagonal toeplitz matrix}.
\newblock \bibinfo{journal}{Journal of Computational and Applied Mathematics}
  \bibinfo{volume}{278}, \bibinfo{pages}{12--18}.
\bibitem[{Yang et~al.(2015)Yang, Wang and Xiang}]{YANG201523}
\bibinfo{author}{Yang, B.}, \bibinfo{author}{Wang, C.}, \bibinfo{author}{Xiang,
  A.}, \bibinfo{year}{2015}.
\newblock \bibinfo{title}{Reversibility of general 1d linear cellular automata
  over the binary field z2 under null boundary conditions}.
\newblock \bibinfo{journal}{Information Sciences} \bibinfo{volume}{324},
  \bibinfo{pages}{23--31}.
\newblock \URLprefix
  \url{https://www.sciencedirect.com/science/article/pii/S0020025515004806},
  \DOIprefix\doi{https://doi.org/10.1016/j.ins.2015.06.048}.
\bibitem[{Zhao and Huang(2008)}]{zhao2008inverse}
\bibinfo{author}{Zhao, X.L.}, \bibinfo{author}{Huang, T.Z.},
  \bibinfo{year}{2008}.
\newblock \bibinfo{title}{On the inverse of a general pentadiagonal matrix}.
\newblock \bibinfo{journal}{Applied Mathematics and Computation}
  \bibinfo{volume}{202}, \bibinfo{pages}{639--646}.

\end{thebibliography}







\end{document}